\documentclass[reqno,11pt]{amsart}
\usepackage{amsmath}
\usepackage{amsthm}
\usepackage{amssymb}
\usepackage{epsfig}
\usepackage[all,dvips]{xy}
\usepackage{tikz}
\usepackage{enumerate}

\newcommand{\F}{\mathbb{F}}
\renewcommand{\b}{\mathbb{B}}
\newcommand{\W}{\mathbb{W}}

\newcommand{\mf}[1]{\ensuremath{\mathfrak #1}}
\newcommand{\pr}[2]{\ensuremath{(#1 \diamond #2)}}

\newcommand{\D}{\mathcal{D}}
\newcommand{\A}{\mathcal{A}}
\newcommand{\B}{\mathcal{B}}
\newcommand{\C}{\mathcal{C}}
\newcommand{\bn}{\ensuremath{\beta}}

\renewcommand{\d}{\partial}

\theoremstyle{plain}
\newtheorem{thm}{Theorem}[section]

\newtheorem{cor}[thm]{Corollary}
\newtheorem{lem}[thm]{Lemma}
\newtheorem{prop}[thm]{Proposition}
\theoremstyle{definition}
\newtheorem{defn}[thm]{Definition}

\newtheorem{example}[thm]{Example}

\theoremstyle{remark}

\title{Homology of the Boolean Complex}

\author{K\'{a}ri Ragnarsson}

\author[Bridget Eileen Tenner]{Bridget Eileen Tenner$^\dagger$}
\thanks{$^\dagger$ Research partially supported by a DePaul University Faculty Summer Research Grant.} 

\address{Department of Mathematical Sciences, DePaul University, Chicago, Illinois}
\email{kragnars@math.depaul.edu, bridget@math.depaul.edu}


\begin{document}

\begin{abstract}

We construct and analyze an explicit basis for the homology of the boolean complex of a finite simple graph.  This provides a combinatorial description of the spheres in the wedge sum representing the homotopy type of the complex.  We assign a set of derangements to any finite simple graph.  For each derangement, we construct a corresponding element in the homology of the complex, and the collection of these elements forms a basis for the homology of the boolean complex.  In this manner, the spheres in the wedge sum describing the homotopy type of the complex can be represented by a set of derangements.

We give an explicit, closed-form description of the derangements that can be obtained from any finite simple graph, and compute this set for several families of graphs.  In the cases of complete graphs and Ferrers graphs, these calculations give bijective proofs of previously obtained enumerative results.\\

\end{abstract}

\maketitle

\section{Introduction}
In \cite{ragnarsson-tenner}, we developed the boolean complex of a finitely generated Coxeter system, based on work by the second author in \cite{tenner}, and analyzed its homotopy type. The boolean complex of a Coxeter system can be constructed from the unlabeled Coxeter graph of the system. We take this as a starting point and talk about the boolean complex $\Delta(G)$  of a finite simple graph $G$, keeping in mind that any such graph can be the unlabeled Coxeter graph of a Coxeter system. Our main result in \cite{ragnarsson-tenner} was to show that the boolean complex of any finite simple graph has the homotopy type of a wedge of spheres. More precisely, there is a homotopy equivalence
\begin{equation} \label{eq:Intro}
 |\Delta(G)| \simeq \bigvee_{i=1}^{\bn(G)}  S^{|G| - 1},
\end{equation}
where the \emph{boolean number} $\bn(G)$ is a graph invariant which can be calculated recursively using edge operations. That this complex is homotopy equivalent to a wedge of spheres was also shown by Jonsson and Welker \cite{jonsson}, who work in a more general context and do not address the question of the number of spheres in this particular setting. In the special case where $G$ is a complete graph, $\Delta(G)$ is a complex of injective words, and Farmer \cite{farmer} and Bj\"orner--Wachs \cite{bjorner-wachs} each showed that this complex has the homotopy type of a wedge of spheres. 
The number of spheres in this case was enumerated by Reiner--Webb in \cite{reiner-webb}.

The boolean number is an interesting graph invariant, several examples of which were calculated in \cite{ragnarsson-tenner} and \cite{ckrt}. For some families of graphs the boolean number has intriguing enumerative properties. One notable example are the complete graphs, for which we showed that $\bn(K_n)$ is the number of derangements of an $n$-element set in \cite{ragnarsson-tenner}, recovering the enumerative result of Reiner--Webb from \cite{reiner-webb} mentioned above. Another example are the Ferrers graphs of staircase shapes, whose boolean numbers were shown in \cite{ckrt} to be the Genocchi numbers of the second kind, which also enumerate derangements with alternating excedances.

In this paper we give bijective proofs of these enumerative results and present a combinatorial description of the spheres appearing in the wedge sum in \eqref{eq:Intro}.  This meaning is not apparent from the proofs in the papers cited above. Each of these proofs uses either discrete Morse theory (\cite{ragnarsson-tenner}) or shellability (\cite{farmer, jonsson, bjorner-wachs}) to collapse most of the respective complexes down to a point, leaving only cells of maximal dimension connected to a single point, which results in a homotopy type of a wedge of spheres. However, which maximal cells remain at the end of the collapsing process depends on a series of choices, and the cells themselves do not form spheres within $\Delta(G)$. Therefore this process does not shed much light on what role the spheres in \eqref{eq:Intro} really play.

Our approach in this paper is via the homology of the boolean complex. Of course, since $\Delta(G)$ has the homotopy type of a group of $(|G|-1)$-dimensional spheres, we know that the reduced homology of $\Delta(G)$ is a graded group of dimension $\bn(G)$ concentrated in degree $|G|-1$. Rather than calculating the homology groups, our goal is to find, and analyze, an explicit basis for the homology of $\Delta(G)$. Each element in this basis is a formal sum of cells in $\Delta(G)$, and taking the union of the closure of these cells in $\Delta(G)$ gives a subcomplex that has the homotopy type of a sphere. (We actually take homology with $\F_2$ coefficients but the basic principle is the same.) The homological generators of the boolean complex are therefore good representatives for the spheres in \eqref{eq:Intro}.

The route from graphs to homological generators takes a detour through derangements. Given a graph $G$ we recursively construct a set of derangements of the vertex set of $G$ using the same edge operations used for calculating the boolean number of a graph. To each of these derangements we associate an element in the homology of $\Delta(G)$, and the central result of the paper is Theorem \ref{thm:main} which says that the outcome of these two steps is a basis for the homology of $\Delta(G)$. Thus we can regard the derangements associated to a graph $G$ as combinatorial representatives for the spheres in the wedge sum describing the homotopy type of $\Delta(G)$. Another important result is Theorem \ref{thm:criterion} where we give an explicit, closed-form description of the set of derangements associated to $G$. When $G$ is a complete graph, this is the set of all derangements, and when $G$ is the Ferrers graph associated to a staircase shape, this is the set of permutations of alternating excedances. Thus we obtain bijective proofs of the enumerative results mentioned earlier.

\section{Background}\label{section:background}

\subsection{Boolean complexes}
Boolean complexes were introduced in \cite{ragnarsson-tenner} as cell complexes associated to Coxeter systems. In fact, as noted in \cite{ragnarsson-tenner}, the boolean complex of a Coxeter system only depends on its unlabeled Coxeter graph, and therefore one can equally well think of boolean complexes associated to finite simple graphs. We shall focus on the graph perspective here, referring the reader to \cite{ragnarsson-tenner} for the relationship to Coxeter systems.

Given a graph $G$ with vertex set $V(G)$, let $\W(G)$ be set of words on letters in $V(G)$ without repetition, ordered by subword inclusion. Form an equivalence relation $\sim$ on $\W(G)$ by saying that two letters commute if and only if they are not adjacent in $G$. More precisely, if
$a_1,\ldots,a_k,s,t,b_1,\ldots,b_\ell$ are distinct vertices in $G$, and $s$ and $t$ are not adjacent, then
 \[ a_1\cdots a_k s t b_1 \cdots b_\ell \sim a_1\cdots a_k t s b_1 \cdots b_\ell \]
in $\W(G)$, and $\sim$ is the equivalence relation generated by this condition. Let $\b(G)$ be the set of equivalence classes in $\W(G)$ with respect to $\sim$, with partial order induced by the order relation on $\W(G)$. We call $\b(G)$ the \emph{boolean poset} of $G$ (this was called the boolean ideal in \cite{ragnarsson-tenner}), and say that an element in $\W(G)$ is a \emph{word representative} for its equivalence class in $\b(G)$.

The boolean poset is ranked by word length, where we adopt the convention that a word of length $k+1$ has rank $k$ (so the empty word has rank $-1$). For each integer $k$, let $\b_k(G) \subseteq \b(G)$ be the subset of elements of rank $k$. It is not hard to see that $\b(G)$ is a simplicial poset, meaning that the principal (lower) order ideal of every element is isomorphic to a boolean algebra. Consequently (\cite{bjorner}), there is an associated regular cell complex $\Delta(G)$ that has $\b(G)$ as face poset, where the dimension of each cell in $\Delta(G)$ equals the rank of the corresponding element in $\b(G)$. The empty word corresponds to the empty cell and will henceforth be ignored. We refer to $\Delta(G)$ as the \emph{boolean complex} of $G$.

Note that, although the closure of each cell in $\Delta(G)$ is a simplex, $\Delta(G)$ is not a simplicial complex. This is because a cell in $\Delta(G)$ is not determined by the $0$-cells in its closure.  An example of this, described in full detail in \cite{ragnarsson-tenner}, is the graph consisting of two vertices and an edge between them.  Instead $\Delta(G)$ is a $\Delta$-complex, as defined in \cite{hatcher}.

\subsection{Homotopy type of boolean complexes}
One obtains a geometric realization $|\Delta(G)|$ in the standard way: take a geometric simplex of dimension $k$ for each $k$-cell in $\Delta(G)$, and glue simplices together according to the face poset. By the homotopy type of the boolean complex we mean the homotopy type of $|\Delta(G)|$. The goal of this section is to recall the main result in \cite{ragnarsson-tenner}, stated below as Theorem \ref{thm:rt}, which shows that the boolean complex of a graph $G$ with $n$ vertices has the homotopy type of a wedge of $(n-1)$-spheres. In \cite{ragnarsson-tenner} we gave a recursive formula for computing the number of spheres occurring in the wedge sum. This number is a graph invariant, which we denote $\bn(G)$ and call the \emph{boolean number} of $G$. The recursion was given in terms of graph operations, defined as follows. 

\begin{defn} \label{def:EdgeOperations}
Let $G$ be a finite simple graph and $e$ an edge in $G$.
\begin{itemize}
\item \emph{Deletion:} $G-e$ is the graph obtained by deleting the edge $e$.
\item \emph{Simple contraction:} $G/e$ is the graph obtained by contracting the edge $e$ and then removing all loops and redundant edges.
\item \emph{Extraction:} $G-[e]$ is the graph obtained by removing the edge $e$ and its incident vertices.
\end{itemize}
\end{defn}

In the statement of the theorem, and throughout the paper, we let $|G|$ denote the number of vertices in a finite graph $G$, and we let $\delta_n$ denote the graph with $n$ vertices and no edges. Furthermore, we use the symbol $\simeq$ to denote homotopy equivalence, and write $b \cdot S^r$ for a wedge sum of $b$ spheres of dimension $r$. (In particular $0 \cdot S^r$ is a point.)

\begin{thm}[\cite{ragnarsson-tenner}]\label{thm:rt}
For every nonempty, finite simple graph $G$, there is an integer $\bn(G)$ so that
\begin{equation*}
|\Delta(G)| \simeq \bn(G) \cdot S^{|G|-1}.
\end{equation*}
Moreover, the values $\bn(G)$ can be computed using the recursive formula
\begin{equation*}
  \bn(G) = \bn(G - e) + \bn(G/e) + \bn(G -[e]),
\end{equation*}
if $e$ is an edge in $G$ such that $G -[e]$ is nonempty, with initial conditions
$$\bn(K_2) = 1\text{\ and\ } \bn(\delta_n) = 0,$$
where $K_2$ is the complete graph on two vertices.
\end{thm}

The following results are also useful.

\begin{cor}[\cite{ragnarsson-tenner}]\label{cor:disjoint point}
A finite simple graph $G$ has an isolated vertex if and only if $\bn(G)=0$.  That is, the center of a Coxeter group contains a generator of the group if and only if the group's boolean complex is contractible.
\end{cor}

\begin{prop}[\cite{ragnarsson-tenner}]\label{prop:disjoint union}
If $G = H_1 \sqcup H_2$ for graphs $H_1$ and $H_2$, then
\[ \Delta(G) = \Delta(H_1) * \Delta(H_2), \]
where $*$ denotes simplicial join, and consequently
\[ |\Delta(G)| \simeq \bn(H_1)\bn(H_2) \cdot S^{|H_1| + |H_2| -1}. \]
In particular, $\bn(G) = \bn(H_1) \bn(H_2)$.
\end{prop}

\subsection{Homology of boolean complexes}
In this subsection we establish the notation and conventions for discussing and computing homology in this paper. We assume the reader is familiar with the procedure for computing the homology of a $\Delta$-complex, which is well known and described, for example, in \cite{hatcher}. Since we know a priori that the boolean complex of a finite simple graph has the homotopy type of a wedge of spheres, this process can be greatly simplified for boolean complexes. To further simplify matters, we always take homology with $\F_2$-coefficients --- thus avoiding issues of orientation --- and write $H_*(-)$ for $H_*(-,\F_2)$ from now on. 

Let $G$ be a finite simple graph with vertex set $V$. By Theorem \ref{thm:rt}, $\tilde{H}^*(\Delta(G))$ is a graded $\F_2$-vector space of dimension $\bn(G)$, concentrated in degree $|G|-1$. Since $\Delta(G)$ has no cells of dimension above $|G|-1$, this means we can compute the homology of $\Delta(G)$ simply as the kernel of the differential $\d^G \colon C_{|G|-1}(\Delta(G)) \to C_{|G|-2}(\Delta(G)),$ where $C_k(\Delta(G))$ denotes the free $\F_2$-vector space with basis $\b_k(G)$. In our setting, this differential can be described as a sum
\[ \d^G := \sum_{v \in V} \d^G_{v} \colon C_{|G|-1}(\Delta(G)) \to C_{|G|-2}(\Delta(G))\, , \]
where, for $v \in V$, the homomorphism 
\[ \d^G_{v} \colon C_{|G|-1}(\Delta(G)) \to C_{|G|-2}(\Delta(G)), \]
sends a basis element $\sigma \in \b_{|G|-1}(G)$ to the string obtained by deleting $v$. When the graph $G$ is clear from the context, we will drop the superscript $G$ and just write $\d$ and $\d_v$ for $\d^G$ and $\d^G_v$, respectively. We refer to an element in the kernel of $d$ as a \emph{homological cycle}.

Loosely speaking, the differential $\d$ sends a cell in $\bn(G)_{|G|-1}$ to its boundary, and a homological cycle is then a formal sum of cells whose boundaries cancel. For a nonzero homological cycle, the closure of the corresponding union of $(|G|-1)$-simplices in $|\Delta(G)|$ forms a $(|G|-1)$-sphere embedded in $|\Delta(G)|$. Thus finding a canonical set of generators for the homology of a boolean complex amounts to giving a combinatorial meaning to the spheres appearing in the wedge-sum representation of its homotopy type.

We will repeatedly use the following lemma to identify homological cycles. The proof is obvious.

\begin{lem} \label{lem:dv}
Let $G$ be a finite simple graph with vertex set $V(G)$. For $x \in C_{|G|-1}(G)$ the following are equivalent:
\begin{enumerate}
 \item $\d(x) = 0$, and
 \item $\d_v(x) = 0$ for each $v \in V(G)$.
\end{enumerate}
\end{lem}

\subsection{String concatenation product} \label{subsec:Product}
For a finite graph $G$, the chain complex $C_*(\Delta(G))$ has a product structure given by word concatenation. This is defined on basis elements as follows. Let $\sigma$ and $\tau$ be word representatives for generators $\bar{\sigma} \in C_k(\Delta(G))$ and $\bar{\tau} \in C_\ell(\Delta(G))$. Thus $\sigma$ is a word on $k +1 $ vertices of $G$, and $\bar{\sigma}$ is the corresponding element in $\b_k(G) \subset  C_k(\Delta(G))$, and similarly with $\tau$. We define the concatenation product by
 \[  \bar{\sigma} \bar{\tau} := 
        \begin{cases}
           \overline{\sigma \tau} \in C_{k+l+1}(\Delta(G)), &\text{if $\sigma$ and $\tau$ have no common letters,}\\
           0, &\text{if $\sigma$ and $\tau$ have common letters.}                      
        \end{cases}  
 \]
It is not hard to check that this is a well-defined, associative operation. Extending linearly, we obtain the string concatenation product sending general elements $x \in C_k(\Delta(G))$ and $y \in C_\ell(\Delta(G))$ to $xy \in C_{k+l+1}(\Delta(G))$.

\subsection{Collapsing maps} \label{subsec:Collapse}
Here we discuss the collapsing map of boolean complexes induced by an inclusion of graphs with the same vertex set, and the effect of this map in homology.

\begin{defn}
For a finite set $V$, let $K_V$ be the complete graph with vertex set $V$. When the finite set is $\{1,\ldots,n\}$, write $K_n$ for short.
\end{defn}

Notice that for a complete graph $K_V$, the boolean poset is the same as the poset of injective words: $\b(K_V) = \W(K_V)$. If $G$ is another graph with vertex set $V$, then by definition we have a projection of posets $\W(G) \to \b(G)$. Since $\W(G) = \W(K_V)$, we can interpret this as a map of boolean posets
 \[ \pi_G \colon \b(K_V) \twoheadrightarrow \b(G),\]
which we call a \emph{collapsing map}. This collapsing map induces a map on boolean complexes $\pi_G \colon \Delta(K_V) \to \Delta(G)$ and consequently a map of chain complexes 
\[ {\pi_G}_* \colon C_*(\Delta(K_V)) \to C_*(\Delta(G)). \] 
Being a map of chain complexes means that ${\pi_G}_*$ is a map of graded groups that respects the differential (that is, $\partial_*^G \circ {\pi_G}_* = {\pi_G}_* \circ \partial_*^{K_V}$), and this implies that we have an induced map in homology
 \[ {\pi_G}_* \colon H_*(\Delta(K_V)) \to H_*(\Delta(G)), \]
which we also refer to as the collapsing map.

More generally, when $G$ is a subgraph of a graph $G'$, and both graphs have the same vertex set $V$, there is a collapsing map of boolean posets $\pi_G^{G'} \colon \b(G') \twoheadrightarrow \b(G)$ which satisfies $\pi_G^{G'} \circ \pi_{G'} = \pi_{G}$. This map also induces a map of boolean complexes and a map in homology ${\pi_G^{G'}}_* \colon H_*(\Delta(G')) \to H_*(\Delta(G))$. 

When the graphs involved are clear from the context and there is no danger of confusion, we will write $\pi$ instead of $\pi_G$ or $\pi_G^{G'}$.

\section{From graphs to derangements} \label{sec:GtoD}

In this section we describe an algorithm that constructs a set of derangements of the vertex set of a finite simple graph $G$. The algorithm is recursive, using the same edge operations involved in the recursive computation of the boolean number of the graph, and as a consequence the number of derangements produced is equal to the boolean number of the graph. We also give an explicit, closed-form description of the resulting set of derangements in Theorem \ref{thm:criterion}. 

In Section~\ref{sec:DtoH} we will show how derangements of the vertex set of a graph give rise to homology cycles for the boolean complex of the graph, and in Section~\ref{sec:MainThm} we prove that the homology cycles coming from the derangements constructed in this section form a basis for the homology of the boolean complex. Therefore the derangements constructed here are key to understanding the combinatorial meaning of the spheres representing the homotopy type of the boolean complex.

\begin{defn}
A \emph{derangement} of a finite set $V$ is a permutation of $V$ that has no fixed points.  The set of derangements of $V$ is denoted by $\D_V$.  When $V$ is the finite set $\{1,\dots,n\}$ we may write $\D_n$ for $\D_V$.  The \emph{derangement number} $d_n$ is the cardinality of $\D_n$.
\end{defn}

It is useful here to write permutations in cycle notation.  For example, $(134)(26)(587)$ is the map $1\mapsto3$, $2\mapsto6$, $3\mapsto4$, $4\mapsto1$, $5\mapsto8$, $6\mapsto2$, $7\mapsto5$, and $8\mapsto7$.  Thus a derangement is a permutation which can be written as a product of disjoint cycles, all of which have length at least two.  For small values of $n$, the derangements of $\{1, \ldots, n\}$ and the derangement numbers $d_n$ are given in Table~\ref{table:derangements}.

\begin{defn}
Let $V$ be a linearly ordered finite set. A derangement of $V$ is written in \emph{standard cycle form} if it is written as a product of disjoint cycles so that the minimum element of a cycle appears as the leftmost letter in that cycle, and the cycles are arranged from left to right in increasing values of minimum letters.
\end{defn}

\begin{example}
The permutation $(134)(26)(587)$ is written in standard cycle form, while the alternative representations $(134)(587)(26)$ and $(341)(26)(587)$ are not.
\end{example}

\begin{table}[htbp]\begin{center}
\begin{tabular}{c|l|c}
$n$ & $\mathcal{D}_n = $ Derangements of $\{1, \ldots, n\}$ & $d_n$\\
\hline
1 & none & 0\\
2 & $(12)$ & 1\\
3 & $(123)$, $(132)$ & 2\\
4 & $(1234)$, $(1243)$, $(1324)$, $(1342)$, $(1423)$, $(1432)$, $(12)(34)$, $(13)(24)$, $(14)(23)$ & 9
\end{tabular}
\end{center}
\caption{Derangements and derangement numbers for small cases.}\label{table:derangements}
\end{table}

The derangement-producing algorithm described below assumes a linear ordering of the vertex set, and its output depends on the order. Formally, we are therefore recursively defining subsets $\D(G, \le) \subseteq \D_{V(G)}$ and a map
 \[ (G,\leq) \longmapsto \D(G,\leq), \]
where $V(G)$ is the vertex set of $G$, and $\leq$ is a linear order on $V(G)$. Such a pair $(G,\leq)$ is usually referred to as an \emph{ordered graph}. In the remainder of the paper we will abuse notation by referring to an ordered graph $G$ instead of $(G, \le)$, and $\D(G)$ instead of $\D(G, \le)$, taking the linear ordering to be understood.  This will not cause any confusion since we only consider one linear ordering for a given graph. In our examples, vertices will be labeled by integers and the ordering is clear.   Note, however, that the derangements in the set $\D(G)$ depend on the ordering of $G$.  That is, a different initial labeling of the graph will yield different derangements.  Thus, a cleverly chosen labeling can be the key to proving certain characteristics of the set of derangements $\D(G)$.

There are three components to the recursive definition of $\D(G)$, including the initial conditions, and these are highlighted as three separate bullet points in the following discussion.

Given a nonempty finite simple ordered graph $G$ with maximal vertex $t$, the initial conditions are as follows.
\begin{itemize}
\item If $V(G) = \{s,t\}$ and there is an edge between these vertices, then $\D(G) := \{(st)\}$.
\item If $t$ is an isolated vertex, then set $\D(G) := \emptyset$.
\end{itemize}

\begin{defn}
Let $G$ be a finite simple ordered graph. If $t$ is the maximal non-isolated vertex of $G$ and $s$ is maximal among vertices adjacent to $t$ then the \emph{maximal edge} of $G$ is the edge $\{s,t\}$.
\end{defn}

In the recursive step we perform edge operations on the maximal edge of $G$.  First, we must explain how, for a finite simple  ordered graph $G$ and an edge $e = \{s,t\}$ in $G$, applying the three edge operations (deletion, simple contraction, and extraction) to $G$ behaves on ordered graphs, and how permutations of the vertex sets of the resulting graphs, $G-e$, $G/e$ and $G-[e]$ give rise to permutations of the vertex set of $G$. 

\begin{quote}
\textbf{Deletion:} $G-e$ has the same vertex set as $G$ and is given the same ordering of vertices. A permutation of $V(G-e)$ is a permutation of $V(G)$.  \\
\textbf{Simple contraction:} Let $x$ be the vertex in $G/e$ obtained by contracting the edge $e$. We obtain a linear ordering on $V(G/e)$ by letting $x$ take the place of $s$ in the ordering on $V(G)$. Given a permutation $w$ of $V(G/e)$, let $w_{st}$ be the permutation of $V(G)$ obtained by writing $w$ in cycle notation, and replacing $x$ by $st$ in the appropriate cycle. \\
\textbf{Extraction:} $V(G - [e])$ is a subset of $V(G)$ and the linear ordering of $V(G)$ restricts to a linear ordering of $V(G-[e])$. Given a permutation $w$ of $V(G-[e])$, a permutation $w\sqcup (st)$ on $V(G)$ is given by applying $w$ to any element of $V(G-[e]) = V(G)\setminus \{s,t\}$ and the transposition $(st)$ to the elements $s$ and $t$. 
\end{quote}

We are now able to state the recursive condition for computing $\D(G)$. As before, $t$ is the maximal vertex of $G$ in this discussion.

\begin{itemize}
 \item If $t$ is not an isolated vertex in $G$ and $e = \{s,t\}$ is the maximal edge in $G$, and $G$ has at least three vertices, then set
\begin{equation}\label{eq:recursive D}
\D(G) := \D(G - e) \cup \{ w_{st} : w \in \D(G/e) \} \cup \{ w\sqcup (st) : w \in \D(G - [e])\}.
\end{equation}
\end{itemize}

We record the following properties of $\D(G)$ which are preserved throughout the recursive construction.
\begin{lem} \label{lem:DforDerangment} \label{lem:Adjacent}
Let $G$ be a finite simple ordered graph and let $t$ be the maximal vertex of $G$. If $w \in \D(G)$ then $w$ is a derangement and $t$ is adjacent to $w^{-1}(t)$ in $G$.
\end{lem}

Also observe that if $G$ has any isolated vertex (maximal or otherwise), then $\D(G) = \emptyset$ because the isolated vertex will be maximal at some step in the iteration, and the second initial condition will apply. This is a special case of the following result.

\begin{prop}\label{prop:distinct derangements}
For a nonempty finite simple ordered graph $G$,
 \[ |\D(G)| = \bn(G). \]
\end{prop}
\begin{proof}
The two numbers $|\D(G)|$ and $\bn(G)$ agree at the initial conditions in the recursive definition for $\D(G)$. The result follows if we can show that they also satisfy the same recurrence relation. We prove this by showing that the three sets on the right-hand side of equation~\eqref{eq:recursive D} are disjoint.

A derangement $v_{st}$ in the second set satisfies $v_{st}(s) = t$ and $v_{st}(t) \neq s$, while a derangement $u\sqcup (st)$ in the third set satisfies $(u\sqcup (st))(s) = t$ and $(u\sqcup (st))(t) = s$. Therefore the second and third sets are disjoint. 
Since $t$ is the maximal vertex in $G-e$ and $s$ is not adjacent to $t$ in $G-e$, Lemma \ref{lem:Adjacent} implies that $w(s) \neq t$ for $w \in \D(G-e)$. Consequently the first set is disjoint from the other two.
\end{proof}

We now present an example of an ordered graph $G$ and show how to compute $\D(G)$.  In practice, contracting two vertices $s$ and $t$ acts in the linear ordering as the absorption of the larger vertex into the smaller vertex.  It is convenient to name the new vertex by the concatenation ``$st$'' of the names of the previous two.  The position of such a vertex in the linear order of the new graph is dictated by the first letter in the concatenation.  This convention is used throughout the following example. 

\begin{example}
Let $G$ be the graph given below, with vertices ordered by their labels.
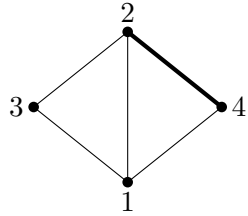
\begin{center}
\begin{tikzpicture}
\draw (0,1) -- (1.25,0); \draw (0,1) -- (1.25,2); \draw (1.25,0) -- (2.5,1); \draw[ultra thick] (1.25,2) -- (2.5,1); \draw (1.25,0) -- (1.25, 2);
\fill[black] (0,1) circle (2pt) node[left] {3}; \fill[black] (2.5,1) circle (2pt) node[right] {4};
\fill[black] (1.25,0) circle (2pt) node[below] {1}; \fill[black] (1.25,2) circle (2pt) node[above] {2};
\end{tikzpicture}
\end{center}
Throughout this example, the maximal edge used in the recursion will be drawn more thickly than other edges.  Let $e$ be the maximal edge in the initial graph.  The graphs $G-e$, $G/e$, and $G-[e]$ are given below, along with their accompanying labelings.
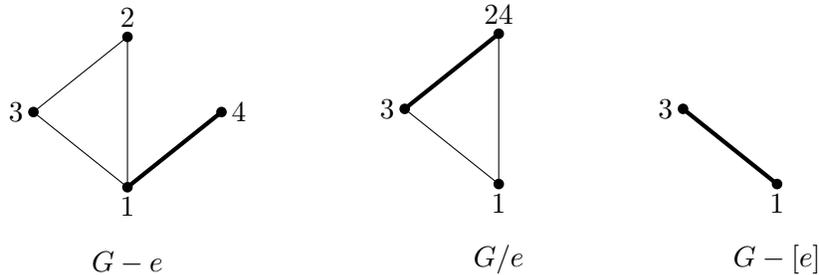
\begin{center}
\begin{tikzpicture}
\draw (0,1) -- (1.25,0); \draw (0,1) -- (1.25,2); \draw[ultra thick] (1.25,0) -- (2.5,1); \draw (1.25,0) -- (1.25, 2);
\fill[black] (0,1) circle (2pt) node[left] {3}; \fill[black] (2.5,1) circle (2pt) node[right] {4};
\fill[black] (1.25,0) circle (2pt) node[below] {1}; \fill[black] (1.25,2) circle (2pt) node[above] {2};
\draw (1.25,-1) node {$G-e$};
\end{tikzpicture}
\hspace{.5in}
\begin{tikzpicture}
\draw (0,1) -- (1.25,0); \draw[ultra thick] (0,1) -- (1.25,2); \draw (1.25,0) -- (1.25, 2);
\fill[black] (0,1) circle (2pt) node[left] {3};
\fill[black] (1.25,0) circle (2pt) node[below] {1}; \fill[black] (1.25,2) circle (2pt) node[above] {24};
\draw (1.25,-1) node {$G/e$};
\end{tikzpicture}
\hspace{.5in}
\begin{tikzpicture}
\draw[ultra thick] (0,1) -- (1.25,0);
\fill[black] (0,1) circle (2pt) node[left] {3};
\fill[black] (1.25,0) circle (2pt) node[below] {1};
\draw (1.25,-1) node {$G-[e]$};
\end{tikzpicture}
\end{center}
In the third of these graphs, $G-[e]$, the only operation not leading to the empty set is extraction, with gives the cycle $(13)$.  Thus, combining this with the original extraction operation, this contributes the derangement $(13)(24)$ to $\D(G)$.  In the second graph, $G/e$, extraction of the marked edge $e'$ yields a contribution of $\emptyset$.  Thus we need only consider deletion and simple contraction of this edge, as follows.
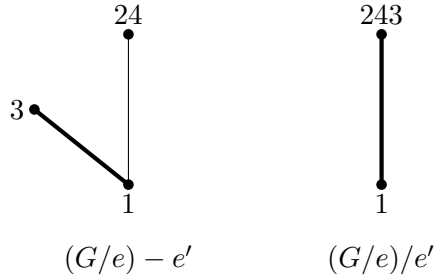
\begin{center}
\begin{tikzpicture}
\draw[ultra thick] (0,1) -- (1.25,0); \draw (1.25,0) -- (1.25, 2);
\fill[black] (0,1) circle (2pt) node[left] {3};
\fill[black] (1.25,0) circle (2pt) node[below] {1}; \fill[black] (1.25,2) circle (2pt) node[above] {24};
\draw (1.25,-1) node {$(G/e)-e'$};
\end{tikzpicture}
\hspace{.5in}
\begin{tikzpicture}
\draw[ultra thick] (1.25,0) -- (1.25, 2);
\fill[black] (1.25,0) circle (2pt) node[below] {1}; \fill[black] (1.25,2) circle (2pt) node[above] {243};
\draw (1.25,-1) node {$(G/e)/e'$};
\end{tikzpicture}
\end{center}
Iterations of this process shows that these two graphs and labelings contribute the derangements $(1324)$ and $(1243)$, respectively.  Finally, the graph $G-e$ yields the derangements $(14)(23)$ and $(1423)$.

Therefore, in this example, $\D(G) = \{(14)(23), (1423), (1324), (1243), (13)(24)\}$.
\end{example}

There is an explicit criterion, given in Theorem \ref{thm:criterion} below, for deciding whether a given derangement $w$ belongs to $\D(G)$ based on connectivity properties of the ordered graph $G$. To describe this criterion we must introduce the following notions.

\begin{defn} \label{def:Canopy}
Let $G$ be an ordered graph and let $w$ be a permutation of its vertex set. For a vertex $t$ in $G$ set
\[ \rho_w(t) = \{t, w(t), \cdots, w^{k-1}(t) \} \, , \]
where $k$ is the smallest positive integer such that $ w^k(t) \leq t$, and set
\[  \lambda_w(t) = w^{-\ell}(t) \, , \]
where $\ell$ is the smallest positive integer such that $w^{-\ell}(t) \leq t$.
\end{defn}
When $w$ is written in standard cycle form and $t$ is not the smallest element in its cycle, $\lambda_w(t)$ is the first element appearing to the left of $t$ that is smaller than $t$, and $\rho_w(t)$ is the set of elements obtained by starting at $t$ and moving to the right until reaching an element less than $t$. When $t$ is the smallest element in its cycle, $\rho_w(t)$ is the entire set of elements in the cycle of $t$, and $\lambda_w(t) = t$.  This characterizes the smallest element, as noted in the following lemma.

\begin{lem}\label{lem:smallest element}
Given a derangement $w$ of an ordered set $V$, written in standard cycle form, and an element $t \in V$, $\lambda_w(t) \in \rho_w(t)$ if and only if $t$ is the smallest element in its cycle.
\end{lem}

\begin{example}\label{ex:canopy}
Let $w = (13472)(56)$.  Then $\lambda_w$ and $\rho_w$ are given below.
$$\begin{array}{c|c|l}
t & \lambda_w(t) & \rho_w(t)\\
\hline
1 & 1 & \{1,2,3,4,7\}\\
2 & 1 & \{2\}\\
3 & 1 & \{3,4,7\}\\
4 & 3 & \{4,7\}\\
5 & 5 & \{5,6\}\\
6 & 5 & \{6\}\\
7 & 4 & \{7\}
\end{array}$$
\end{example}

We now state and prove the criterion for membership in the set $\D(G)$.  Example~\ref{ex:criterion} depicts an instance of when this criterion is met and when it is not.

\begin{thm} \label{thm:criterion}
Let $G$ be a finite simple ordered graph and let $w$ be a permutation of its vertex set. Then $w \in \D(G)$ if and only if for every vertex $t$ of $G$  
the vertex $\lambda_w(t)$ is adjacent to some vertex in $\rho_w(t)$.
\end{thm}
\begin{proof}
For a finite simple ordered graph $G$ and a permutation $w$ of its vertex set, we will say that $w$ is $G$-valid at a vertex $r$ of $G$ if $\lambda_w(r)$ is adjacent to some vertex in $\rho_w(r)$. We say that $w$ is $G$-valid if it is $G$-valid at $r$ for every vertex $r$ of $G$. Let $\D'(G)$ denote the set of $G$-valid permutations $V(G)$. We will show that $\D'(-)$ satisfies the same recursion and the same initial conditions as $\D(-)$, and hence $\D'(G) = \D(G)$.

First let us observe that the conclusion of Lemma \ref{lem:Adjacent} holds for $\D'(G)$: If $w$ is a permutation of $V(G)$ with $w(r) = r$ for some vertex $r$ then $\lambda_w(r) =r $ and $\rho_w(r) = \{r\}$ so $w$ is not $G$-valid at $r$ (since $G$ is assumed to be simple). Thus any element of $\D'(G)$ is a derangement. Furthermore, if $t$ is the maximal vertex of $G$ and $w \in \D'(G)$ then $\lambda_w(t) = w^{-1}(t)$ and $\rho_w(t) = \{t\}$, so $G$-validity at $t$ implies that $t$ is adjacent to $ w^{-1}(t)$ in $G$.

Now we establish the initial conditions. If $G$ is an ordered graph with vertices $s < t$ and an edge between them, then one easily checks that $\D'(G) =  \{(st)\} = \D(G)$. If the maximal vertex $t$ of $G$ is isolated then no derangement $w$ of $V(G)$ can be $G$-valid at $t$ since always $\rho_w(t) = \{t\}$. Hence $\D'(G) =\emptyset = \D(G)$. 

For the recursive step let $G$ be an ordered graph with at least two edges such that the maximal vertex $t$ is not isolated and let $e = \{s,t\}$ be the maximal edge in $G$. We can write $\D'(G)$ as a disjoint union
\[ \D'(G) = \A(G) \sqcup \B(G) \sqcup \C(G) \, ,\]
where 
\begin{align*}
  \A(G) &= \{ w \in \D'(G) : w(s) \neq t \} \, , \\
  \B(G) &= \{ w \in \D'(G) : w(s) = t, w(t) \neq s \} \, , \text{ and} \\
  \C(G) &= \{ w \in \D'(G) : w(s) = t, w(t) = s \} \, .
\end{align*}
The proof is complete once we establish the following equalities.
\begin{align}
  \A(G) &= \D'(G-e), \label{eq:DG1}\tag{i}\\
  \B(G) &= \{ u_{st} : u \in \D'(G/e) \} \tag{ii} \, , \text{and} \label{eq:DG2}\\
  \C(G) &= \{ v\sqcup (st) : v \in \D'(G-[e]) \} \tag{iii} \, .\label{eq:DG3}
\end{align}

We start by proving \eqref{eq:DG1}. Any $(G-e)$-valid derangement of $V(G-e) = V(G)$ is also $G$-valid, so $\D'(G-e) \subseteq \D'(G)$. Furthermore, for $w \in \D'(G-e)$ we know that $t$ is adjacent to $w^{-1}(t)$ in $G-e$, so $s \neq w^{-1}(t)$. Combining these facts we have $\D'(G-e) \subseteq \A(G)$. 

Now suppose $w \in \A(G)$. If $w \notin \D'(G-e)$ then there is some vertex $r$ at which $w$ is not $(G-e)$-valid. By assumption $w$ is $G$-valid at $r$, so we must have  $\lambda_w(r) = s$ and $t \in \rho_w(r)$. In particular, $s$ is in the same $w$-cycle as $t$ and appears left of $t$ when $w$ is written in standard cycle form. However, we know that $t$ is adjacent to $w^{-1}(t)$, and maximality of $e = \{s,t\}$ and the assumption $w(s) \neq t$ then imply $w^{-1}(t) < s \leq r$. Since $w^{-1}(t)$ appears between $s$ and $t$ when $w$ is written in standard cycle form, this makes the conditions $\lambda_w(r) = s$ and $t \in \rho_w(r)$ incompatible, leading to a contradiction. We conclude that $w \in \D'(G)$, and since this holds for all $w \in \A(G)$ we have $\A(G) \subseteq \D'(G-e)$, completing the proof of \eqref{eq:DG1}.

To prove \eqref{eq:DG3} we first observe that any derangement $w$ of $V(G)$ such that $w(s) = t$ and $w(t) = s$ can be written as $w = v \sqcup (st)$ where $v$ is a derangement of $V(G-[e])$. Clearly $w \in \C(G)$ if and only if $v \in \D'(G-[e])$.

It remains to prove \eqref{eq:DG2}. Any derangement $w$ of $V(G)$ such that $w(s) = t$ and $w(t) \neq s$ can be written as $w =u_{st}$ where $u$ is a derangement of $V(G/e)$. We show that $w$ is $G$-valid if and only if $u$ is $G/e$-valid. First observe that $w$ is always $G$-valid at $t$ since $\lambda_w(t) =s$ and $\rho_w(t) = \{t\}$. Now let $x$ denote the vertex in $G/e$ obtained upon identifying $s$ with $t$. Since the ordering of $G/e$ is obtained by letting $x$ take the place of $s$ in the ordering of $G$ we have an injective, order-preserving map $i \colon V(G/e) \to V(G)$ that sends $x$ to $s$ and any other vertex to itself. For each $r$ in $V(G/e)$ we will show that $G/e$-validity of $u$ at $r$ implies $G$-validity of $w$ at $i(r)$ and that $G$-validity of $w$ (at every vertex) implies $G/e$-validity of $u$ at $r$. Since $i$ has image $V(G)\setminus \{t\}$, and we already know that $w$ is $G$-valid at $t$, it follows that $w$ is $G$-valid if and only if $u$ is $G/e$-valid, completing the proof. To prove the claim one checks that for $r \in V(G/e)$ one has 
\[ \lambda_w(i(r)) = i(\lambda_u(r)) \]
and
\[ \rho_w(i(r)) = \begin{cases} 
                \rho_u(r), &\text{if $x\notin \rho_u(r)$}\, ;\\
                \left(\rho_u(r)\setminus \{x\}\right) \cup \{s,t\}, &\text{if $x\in \rho_u(r)$}\, .
                \end{cases} \]
We consider four cases.

If $\lambda_u(r) \neq x$ and $x \notin \rho_u(r)$ then $\lambda_w(i(r)) = \lambda_u(r)$ and $\rho_w(i(r)) = \rho_u(r)$ so the result is obvious.

If $\lambda_u(r) \neq x$ and $x \in \rho_u(r)$ then $\lambda_w(i(r)) = \lambda_u(r)$ and $\rho_w(i(r)) = \left(\rho_u(r)\setminus \{x\}\right) \cup \{s,t\}$. The result now follows from the fact that $\lambda_u(r)$ is adjacent to $x$ in $G/e$ if and only if it is adjacent to either $s$ or $t$ in $G$.

If $\lambda_u(r) = x$ and $x \notin \rho_u(r)$ then $\lambda_w(i(r)) =s $ and $\rho_w(i(r)) = \rho_u(r)$. Furthermore, since $\lambda_u(r) \notin \rho_u(r)$ we know that $r$ is not minimal in its $u$-cycle by Lemma~\ref{lem:smallest element}. It follows that $i(r)$ is not minimal in its $w$-cycle; in particular $i(r) > \lambda_w(i(r)) =s$. For any $q \in \rho_w(i(r))$ we have $q \geq i(r)> s$ so maximality of $e = \{s,t\}$ implies that $q$ is not adjacent to $t$. Thus we can deduce that $q$ is adjacent to $x$ in $G/e$ if and only if $q$ is adjacent to $s$ in $G$. It follows that $u$ is $G/e$-valid at $r$ if and only if $w$ is $G$-valid at $i(r)$. 

If $\lambda_u(r) = x$ and $x \in \rho_u(r)$ then $r =x$ and $x$ is minimal in its $u$-cycle by Lemma~\ref{lem:smallest element}. It follows that $i(r) = s$ is minimal in its $w$-cycle. In this case $w$ is automatically $G$-valid at $i(r)$ because $s = \lambda_w(i(r))$ is adjacent to $t \in \rho_w(i(r))$. On the other hand, if we assume that $w$ is $G$-valid then $G$-validity at $w(t)$ implies that $s = \lambda_w(w(t))$ is adjacent to a vertex of $\rho_w(w(t))$ in $G$, and hence $r = x$ is adjacent to the same vertex (which also belongs to $\rho_u(r)$ since this set is the entire $u$-cycle of $r$), proving $G/e$-validity of $u$ at $x$.
\end{proof}

\begin{example}\label{ex:criterion}
Let $G$ be the 7-vertex graph depicted below.
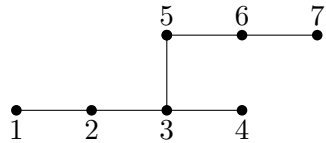
\begin{center}
\begin{tikzpicture}
\draw (1,0) -- (4,0);
\draw (3,0) -- (3,1);
\draw (3,1) -- (5,1);
\foreach \x in {1,2,3,4} {\fill[black] (\x, 0) circle (2pt); \draw (\x,0) node[below] {$\x$};}
\foreach \x in {3,4,5} {\fill[black] (\x, 1) circle (2pt);}
\fill[black] (3,1) circle (2pt);
\draw (3,1) node[above] {$5$};
\draw (4,1) node[above] {$6$};
\draw (5,1) node[above] {$7$};
\end{tikzpicture}
\end{center}
It is easy to check that $(1234)(567)$ is an element of $\D(G)$.  On the other hand, the derangement $(1234567)$ is excluded from $\D(G)$ because $\lambda_{(1234567)}(5) = 4$ and $\rho_{(1234567)}(5) = \{5,6,7\}$, but $4$ is not adjacent to $5$, $6$, or $7$ in the graph.  As another example, the derangement $(13472)(56)$ of Example~\ref{ex:canopy} is excluded from $\D(G)$ because $\lambda_{(13472)(56)}(3) = 1$ and $\rho_{(13472)(56)}(3) = \{3,4,7\}$, and $1$ is not adjacent to $3$, $4$, or $7$.
\end{example}

Theorem \ref{thm:criterion} allows us to describe the derangements induced by the disjoint sum of two ordered graphs in terms of the derangements induced by the parts. In the statement we use the notation $w_1 \sqcup w_2$ to denote the permutation of $V_1 \sqcup V_2$ obtained from a permutation $w_1$ of a set $V_1$ and a permutation $w_2$ of a set $V_2$. 

\begin{cor}
If $G_1$ and $G_2$ are disjoint finite ordered simple graphs then 
\[ \D(G_1 \sqcup G_2) = \{ w_1 \sqcup w_2 : w_1 \in \D(G_1), w_2 \in \D(G_2)  \} \, .\]
\end{cor}

The set $\D(G)$ invites a variety of combinatorial questions.  For example, we can look at those elements in $\D(G)$ whose standard cycle form consists of exactly one cycle.  Let us call this set $\D^1(G) \subseteq \D(G)$.  Note that when extraction is used in the recursive construction of elements of $\D(G)$, additional cycles are created.  Thus $\D^1(G)$ satisfies the recursion
\begin{equation}\label{eq:recursive D^1}
\D^1(G) = \D^1(G - e) \cup \{ w_{st} : w \in \D^1(G/e) \}.
\end{equation}
Some enumerative results about $\D^1(G)$ are stated in the following proposition.

\begin{prop}\label{prop:number with one cycle}
Let $G$ be a finite simple ordered graph.
\begin{enumerate}[(a)]
\item $|\D^1(G)| = 0$ if and only if $G$ is not connected.
\item $|\D^1(G)| = 1$ if and only if $G$ is a tree.
\item If $G$ has exactly one cycle, and that cycle contains $m$ vertices, then $|\D^1(G)| = m-1$.
\end{enumerate}
\end{prop}
\begin{proof}
Statements (a) and (b) follow easily from equation~\eqref{eq:recursive D^1}.  Statement (c) can be proved by induction on the total number of edges in $G$, using the recursion of equation~\eqref{eq:recursive D^1} and the previous statements. 
\end{proof}

It is interesting to note that the results of Proposition~\ref{prop:number with one cycle} are independent of the ordering of the vertex set of $G$. More generally, one can let $\D^k(G) \subseteq \D(G)$ be the subset of derangements with exactly $k$ disjoint cycles. Examples suggest that the sizes of the sets $\D^k(G)$ are independent of the ordering of the vertex set of $G$. In later sections we will show that the elements in $\D(G)$ can be treated as combinatorial representatives for the spheres appearing in the wedge sum description of the homotopy type of $\Delta(G)$ offered by Theorem \ref{thm:rt}. This raises the question of the significance of the number of cycles in $w \in \D(G)$ with respect to the boolean complex $\Delta(G)$.

\section{From derangements to homology}\label{sec:DtoH}

We now describe a map from the set of derangements of the vertex set of a finite ordered graph $G$ to the top-dimensional homology of the boolean complex of $G$. Just as in Section \ref{sec:GtoD}, we fix an ordering $\le$ on $V(G)$, and the output of the map defined here depends on that ordering. As before, we suppress the ``$\le$'' in our notation.  Formally, we are therefore defining a map
 \[ \phi_{G} \colon \D_{V(G)} \longrightarrow H_{|G|-1}(\Delta(G)). \]
When no confusion will arise, this map may be denoted simply as $\phi$.

We initially do this for complete graphs, defining a map $\phi_{K_{V}} : \D_{V} \to C_{|V|-1}(\Delta(K_{V}))$, where $K_{V}$ is the complete graph with vertex set $V$.  We then show in Lemma \ref{lem:IsCycle} that $\phi_{K_{V}}$ actually takes values in $H_{|V|-1}(\Delta(K_{V})) \subseteq C_{|V|-1}(\Delta(K_{V}))$. Finally, the map $\phi_{G}$ is obtained by composing $\phi_{K_{V(G)}}$ with the collapsing map ${\pi_G}_*$, as described in Section~\ref{subsec:Collapse}.

\begin{defn} \label{defn:LieBracket}
Let $G$ be a finite simple graph. For $a \in C_k(\Delta(G))$ and $b \in C_\ell(\Delta(G))$, set
$$\pr{a}{b} := ab+ba \in C_{k+\ell +1}(\Delta(G)),$$
where $ab$ and $ba$ are the string concatenation products (see Section~\ref{subsec:Product}). 
\end{defn}

Consider a linearly ordered set $V$. The following algorithm describes a map $\phi = \phi_{K_{V}} : \mathcal{D}_V \rightarrow C_{|V|-1}(\Delta(K_V))$.  The input is a derangement $w \in \mathcal{D}_V$, written in standard cycle form, and the final output of this procedure is $\phi(w)$.
\begin{quote}
\begin{enumerate}\renewcommand{\labelenumi}{Step \arabic{enumi}.}
\item Between each consecutive pair of letters in each cycle of $w$, insert the symbol $\star$.
\item If there are no $\star$ symbols in the string, then HALT and OUTPUT the string.  Otherwise, determine which symbol $\star$ has the largest right-hand neighbor.
\item Suppose that the symbol $\star$ located in Step 2 is between quantities $Q$ and $R$; that is, it appears as $Q \star R$.  Then replace $Q \star R$ by $\pr{Q}{R}$.
\item GOTO Step 2.
\end{enumerate}
\end{quote}

\begin{example}
Let $w = (134)(26)(587)$.  Applying the above procedure to $w$ gives the following sequence of steps.
\begin{eqnarray*}
&&(1 \star 3 \star 4)(2 \star 6)(5 \star 8 \star 7)\\
&&(1 \star 3 \star 4)(2 \star 6)(\pr{5}{8} \star 7)\\
&&(1 \star 3 \star 4)(2 \star 6)\pr{\pr{5}{8}}{7}\\
&&(1 \star 3 \star 4)\pr{2}{6}\pr{\pr{5}{8}}{7}\\
&&(1 \star \pr{3}{4})\pr{2}{6}\pr{\pr{5}{8}}{7}\\
&&\pr{1}{\pr{3}{4}}\pr{2}{6}\pr{\pr{5}{8}}{7}
\end{eqnarray*}
Thus 
\begin{eqnarray*}
\phi((134)(26)(587)) &=& \pr{1}{\pr{3}{4}}\pr{2}{6}\pr{\pr{5}{8}}{7}\\
&=&\pr{1}{(34+43)}(26+62)\pr{(58+85)}{7}\\
&=&(1(34+43) + (34+43)1)(26+62)((58+85)7 + 7(58+85))\\
&=&(134+143 + 341+431)(26+62)(587+857 + 758+785))\\
&=& 13426587 + 13426857 + 13426758 + 13426785\\
&& + 13462587 + 13462857 + 13462758 + 13462785\\
&& + 14326587 + 14326857 + 14326758 + 14326785\\
&& + 14362587 + 14362857 + 14362758 + 14362785\\
&& + 34126587 + 34126857 + 34126758 + 34126785\\
&& + 34162587 + 34162857 + 34162758 + 34162785\\
&& + 43126587 + 43126857 + 43126758 + 43126785\\
&& + 43162587 + 43162857 + 43162758 + 43162785
\end{eqnarray*}
\end{example}

The following result allows us to regard $\phi$ as a map $\D(K_V) \to H_{|V|-1}(\Delta(K))$, and we do so henceforth without further mention.

\begin{lem} \label{lem:IsCycle}
Let $V$ be finite ordered set. If $w$ is a derangement of $V$, then $\phi_{(K_V)}(w)$ is a homological cycle.
\end{lem}

\begin{proof}
We must show that $\d(\phi(w))$ equals $0$.  By Lemma \ref{lem:dv} this is equivalent to showing that for each $v \in V$ we have $\d_v(\phi(w)) = 0$. Because $w$ has no fixed points, the letter $v$ is in the same cycle as at least one other letter in the standard cycle form of $w$.  Thus $\phi(w)$ contains a product $\pr{v}{x} = \pr{x}{v}$ for some $x$ not containing the letter $v$.  Observe that
$$\d_v(\pr{v}{x}) = \d_v(vx + xv) = x + x = 2x = 0.$$
This, combined with the fact that if $\d_v(x) = 0$ for some $x$ then $\d_v(\pr{x}{y}) = 0$ for all $y$, implies that $\d_v(\phi(w)) = 0$ for all $v \in V$.  Consequently, $\d(\phi(w)) = 0$, so $\phi(w)$ is a homological cycle.
\end{proof}

\begin{defn}
For a finite simple ordered graph $G$, let $\phi_{G}$ be the composite
 \[ \phi_{G} \colon \D(G) \hookrightarrow \D(V(G)) \xrightarrow{\phi_{(K_{V(G)})} } H_{|G|-1}(\Delta(K_{V(G)})) \xrightarrow{{\pi_G}_*} H_{|G|-1}(\Delta(G)).  \]
\end{defn}

\section{The derangement basis}\label{sec:MainThm}

Having set up the maps from graphs to derangements and from derangements to homology in the previous sections, we are now prepared to prove the main result of this paper.  This states that given a finite simple ordered graph $G$, the homological cycles $\{\phi(w) : w \in \D(G)\}$ form a basis for the homology of $\Delta(G)$.  Looking ahead to this result, we make the following definition.

\begin{defn}
Given a finite simple ordered graph $G$, the set $\{\phi(w) : w \in \D(G)\}$ is the \emph{derangement basis} associated to $G$.
\end{defn}

The definition is justified by the following result.

\begin{thm}\label{thm:main}
Given a finite simple ordered graph $G$, the derangement basis associated to $G$ is a basis for $H_{|G|-1}(\Delta(G))$. 
\end{thm}

\begin{proof}
Lemma \ref{lem:IsCycle} shows that the derangement basis consists of homological cycles, and  Proposition~\ref{prop:distinct derangements} shows that $|\D(G)| = \bn(G)$, which is the $\F_2$-rank of $H_{|G|-1}(\Delta(G))$.  Thus it remains to show that the elements of the derangement basis are linearly independent. We prove this by induction, assuming that the result holds for graphs with fewer edges. The base case is trivial to show.

Equation~\eqref{eq:recursive D} and the fact that $\phi$ is injective allow us to rewrite the derangement basis associated to $G$ as
a disjoint union
\begin{equation}\label{eqn:three sets for basis}
\{\phi(w) : w \in \D(G - e)\} \cup \{\phi(w_{st}) : w \in \D(G/e) \} \cup \{\phi(w\sqcup(st)) : w \in \D(G - [e])\}.
\end{equation}
The graphs $G - e$, $G/e$, and $G-[e]$ all have fewer edges than $G$, so by induction their associated derangement bases are linearly independent sets.  Using this fact, it follows that each of the three sets in the union \eqref{eqn:three sets for basis} is a linearly independent set.

Suppose that 
\begin{equation}\label{eq:linear combination}
\sum_{u \in \D(G_e)} \epsilon_u \phi(u) + \sum_{v \in \D(G/e)} \epsilon_v \phi(v_{st}) + \sum_{w \in \D(G-[e])} \epsilon_w \phi(w\sqcup(st)) = 0.
\end{equation}

If we apply the collapsing map $\pi: H_*(\Delta(G)) \rightarrow H_*(\Delta(G-e))$ of Section~\ref{subsec:Collapse} to equation~\eqref{eq:linear combination}, then elements arising from $\D(G/e)$ and $\D(G-[e])$ are mapped to 0, since they can be written as a sum of terms of products involving $st + ts$.  This leaves $\sum_{u \in \D(G_e)} \epsilon_u \phi(u) = 0$.  Elements of $\{\phi(w) : w \in \D(G - e)\}$ are linearly independent, so this implies that $\epsilon_u = 0$ for all $u \in \D(G_e)$.

We must now consider a linear combination
$$\sum_{v \in \D(G/e)} \epsilon_v \phi(v_{st}) + \sum_{w \in \D(G-[e])} \epsilon_w \phi(w\sqcup(st)) = 0.$$
If $x$ is the vertex in $G/e$ obtained by contracting the edge $e$, then $\phi(v_{st})$ is obtained from $\phi(v)$ by replacing $x$ by $st+ts$.  Write $\phi(v_{st}) = \phi(v)_{[st]} + \phi(v)_{[ts]}$, where the former is obtained by replacing $x$ in $\phi(v)$ by $st$ and the latter is obtained by replacing $x$ in $\phi(v)$ by $ts$.  Thus the above equation can be rewritten as
$$\sum_{v \in \D(G/e)} \epsilon_v \left(\phi(v)_{[st]} + \phi(v)_{[ts]}\right) + \sum_{w \in \D(G-[e])} \epsilon_w \phi(w)(st+ts) = 0.$$
If we look at those terms involving $st$, and analogously at those involving $ts$, we see that
\begin{equation}\label{eq:linear combination - two}
\sum_{v \in \D(G/e)} \epsilon_v \phi(v)_{[st]} + \sum_{w \in \D(G-[e])} \epsilon_w \phi(w)st = 0.
\end{equation}
Notice that we can identify the sub-chain complex $C_*(\Delta(G))_{\tilde{s},\tilde{t}} $ of $C_*(\Delta(G))$ generated by strings that do not involve $s$ or $t$ with the sub-chain complex  $C_*(\Delta(G/e))_{\tilde{x}}$ of $C_*(\Delta(G/e))$ generated by strings that do not involve $x$. Applying the differential maps 
$\d_s$ and $\d_t$ to equation~\eqref{eq:linear combination - two}, we get an equation in $C_*(\Delta(G))_{\tilde{s},\tilde{t}} $ which is equivalent to the equation 
\begin{equation}\label{eqn:apply partials}
\d_x \left(\sum_{v \in \D(G/e)} \epsilon_v \phi(v)\right) + \sum_{w \in \D(G-[e])} \epsilon_w \phi(w) = 0,
\end{equation}
in $C_*(\Delta(G/e))_{\tilde{x}}$,
since deleting both $s$ and $t$ from a term in $\phi(v_{st})$, or in $\phi(v)_{[st])}$, is equivalent to deleting $x$ from that term in $\phi(v)$.

Recall that $\d_x\phi(v) = 0$ for all $v \in \D(G/e)$.  Therefore, equation~\eqref{eqn:apply partials} reduces to
$$\sum_{w \in \D(G-[e])} \epsilon_w \phi(w) = 0.$$
Elements of $\{\phi(x) : w \in \D(G-[e])\}$ are linearly independent, so $\epsilon_w = 0$ for all $w \in \D(G-[e])$.  Finally, equation~\eqref{eq:linear combination} now simplifies to
$$\sum_{v \in \D(G/e)} \epsilon_v \phi(v_{st}) = 0,$$
from which it follows that $\epsilon_v = 0$ for all $v \in \D(G/e)$.

Thus all coefficients in equation~\eqref{eq:linear combination} must be 0, and the derangement basis associated to $G$ is linearly independent.
\end{proof}

\section{Examples}\label{section:examples}

We now consider several families $\mathcal{G}$ of graphs, and show that with certain well-chosen orderings of the vertices, the sets $\D(G)$ for $G \in \mathcal{G}$ have particularly nice properties.  The families considered here are complete graphs, Ferrers graphs for staircase shapes, and unlabeled graphs for the classical and affine Coxeter groups.  The results regarding the sets $\D(G)$ support previously obtained results in \cite{ragnarsson-tenner} and \cite{ckrt} for boolean numbers.

\subsection{Complete graphs}

Let $K_n$ denote the complete graph on $n$ vertices.  It was shown in \cite{ragnarsson-tenner}, as well as in \cite{reiner-webb} in terms of injective words, that the boolean number of the complete graph is equal to the derangement number. That is,
\begin{equation}\label{eqn:K_n=d_n}
\beta(K_n) = d_n.
\end{equation}
This equality was proved independently in both  \cite{reiner-webb} and \cite{ragnarsson-tenner} by demonstrating that the sequences $\{\bn(K_n)\}$ and $\{d_n\}$ satisfy the same recurrence relation and have the same initial values.  However, a combinatorial understanding of the identity was lacking.  We now justify equation~\eqref{eqn:K_n=d_n} combinatorially, thus giving a satisfying understanding of the relationship between the spheres in $\Delta(K_n)$ and derangements, by explaining the connection between derangements of $\{1,\ldots, n\}$ and generators of the homology of $\Delta(K_n)$.

\begin{cor} \label{cor:Kn}
For any ordering of the vertices of $K_n$, we have $\D(K_n) = \D_{V(K_n)}$.
\end{cor}

\begin{proof}
The criterion provided in Theorem~\ref{thm:criterion}, together with the fact that all possible edges exist in the complete graph, indicate that every possible derangement can be obtained.  Thus we see bijectively that $\D(K_n) = \D_{V(K_n)}$.
\end{proof}

Note that this could also be shown by a simple counting argument, as follows.  Proposition~\ref{prop:distinct derangements} shows that $\D(K_n)$ contains $\bn(K_n)$ distinct derangements, and $\bn(K_n) = d_n$ by \cite{reiner-webb} and \cite{ragnarsson-tenner}.  Since $\D(G) \subseteq \D_{V(G)}$ for all graphs $G$, and $|\D_{V(G)}| = d_{|V(G)|}$, we see that in fact $\D(K_n) = \D_{V(K_n)}$.

\begin{example}
Applying Theorem~\ref{thm:main} to the complete graph with vertices $\{1, 2, \ldots, n\}$ gives the homology generators described below, along with the corresponding derangements.
$$\begin{array}{r|l|l}
n & \D(K_n)& \text{Derangement basis of } K_n\\
\hline
1 & \emptyset & \emptyset\\
\hline
2 & (12) & \pr{1}{2} = 12+21\\
\hline
3 & (123) & \pr{1}{\pr{2}{3}} = 123+132+231+321\\
& (132) & \pr{\pr{1}{3}}{2} = 132+312+213+231\\
\hline
4 & (1234) & \pr{1}{\pr{2}{\pr{3}{4}}}\\
& & \hspace{.25in} = 1234+1243+1342+1432+2341+2431+3421+4321\\ 
& (1243) & \pr{1}{\pr{\pr{2}{4}}{3}}\\
& & \hspace{.25in} = 1243+1423+1324+1342+2431+4231+3241+3421\\
& (1324) & \pr{\pr{1}{3}}{\pr{2}{4}}\\
& & \hspace{.25in} = 1324+3124+1342+3142+2413+2431+4213+4231\\
& (1342) & \pr{\pr{1}{\pr{3}{4}}}{2}\\
& & \hspace{.25in} = 1342+1432+3412+4312+2134+2143+2341+2431\\
& (1423) & \pr{\pr{1}{4}}{\pr{2}{3}}\\
& & \hspace{.25in} = 1423+4123+1432+4132+2314+2341+3214+3241\\
& (1432) & \pr{\pr{\pr{1}{4}}{3}}{2}\\
& & \hspace{.25in} = 1432+4132+3142+3412+2143+2413+2314+2341\\
& (12)(34) & \pr{1}{2}\pr{3}{4}  = 1234+2134+1243+2143\\
& (13)(24) & \pr{1}{3}\pr{2}{4} = 1324+3124+1342+3142\\
& (14)(23) & \pr{1}{4}\pr{2}{3} = 1423+4123+1432+4132
\end{array}$$
\end{example}

\subsection{Ferrers graphs for staircase shapes}

We now use the main result of this paper to give an explanation of an enumerative result obtained in \cite{ckrt} for certain Ferrers graphs.

\begin{defn}
For $r \ge 1$, the \emph{staircase shape of height $r$} is the Ferrers shape $\sigma_r = (r, r-1, \ldots, 2, 1)$.  The \emph{Ferrers graph} $F_r$ of the staircase shape $\sigma_r$ is the bipartite graph on $2r$ vertices where $r$ vertices describe the rows of $\sigma_r$, $r$ vertices describe the columns of $\sigma_r$, and two vertices are adjacent if there is a square in the corresponding row and column of the shape $\sigma_r$.
\end{defn}

It was shown in \cite{ckrt} that
$$\bn(F_r) = g_r$$
for all $r \ge 1$, where $\{g_r\}$ are the median Genocchi numbers, also called the Genocchi numbers of the second kind.  This is sequence A005439 of \cite{oeis}.

\begin{defn}
A permutation $w \in \mf{S}_{2r}$ has \emph{alternating excedances} if $w(i) > i$ whenever $i$ is odd and $w(i) < i$ whenever $i$ is even.  Such a permutation is always a derangement because $w(i) \neq i$ for all $i$.  When written in standard cycle form, this means that each odd number is followed by a larger number, and each even number is followed by a smaller number.  The set of permutations in $\mf{S}_{2r}$ with alternating excedances will be denoted $AE_{2r}$.
\end{defn}

The median Genocchi number $g_r$ is equal to the number of permutations $w \in \mf{S}_{2r}$ such that $w$ has alternating excedances (see \cite{ehrenborg-steingrimsson}); that is,
$$g_r = |AE_{2r}|.$$

\begin{example}
The second median Genocchi number is $g_2 = 2$, and the permutations in $\mf{S}_4$ with alternating excedances are $(12)(34)$ and $(1342)$.
\end{example}

We now explain why $\bn(F_r) = g_r$ by demonstrating an ordering of the vertices of $F_r$ for which $\D(F_r)$ is equal to the set $AE_{2r}$.  The vertices of $F_r$ arise from the rows and columns of the staircase shape $\sigma_r$, so we can define an ordering on the vertices of $F_r$ in terms of the staircase shape.

\begin{defn}
Given a staircase shape $\sigma_r$, label the rows $1, 3, 5, \ldots, 2r-1$ so that the labels increase as the lengths of the rows decrease.  In contrast, label the columns $2, 4, 6, \ldots, 2r$ so that the labels decrease as the lengths of the columns decrease.  This gives a corresponding ordering on the vertices of $F_r$.
\end{defn}

The labeling on the shape $\sigma_r$ is depicted in Figure~\ref{fig:ferrers labels}.  Therefore, the bipartite graph $F_r$ consists of one set of vertices $\{1,3,5,\ldots, 2r-1\}$ and a second set of vertices $\{2,4,6,\ldots, 2r\}$.  The edges in this graph are exactly those of the form $\{2i+1,2j\}$ where $i < j$.  That is, each odd vertex is connected to every even vertex having a larger label.  Equivalently, each even vertex is connected to every odd vertex having a smaller label.  The correspondingly labeled Ferrers graph for the case $r=4$ is given in Figure~\ref{fig:ferrers graph labels}.

\begin{figure}[htbp]
\begin{center}
\begin{tikzpicture}[scale=.75]
\foreach \x in {1,2,3,4,5}{\draw (0,\x-1) -- (\x,\x-1);}
\draw (0,5) -- (5,5); \draw (0,5) -- (0,0);
\foreach \x in {1,2,3,4,5}{\draw (6-\x,5) -- (6-\x,5-\x);}
\draw (-1,4.5) node {$1$};
\draw (-1,3.5) node {$3$};
\draw (-1,2.5) node {$5$};
\draw (-1,1.5) node {$\vdots$};
\draw (-1,.5) node {$2r-1$};
\draw (.5,5.5) node {$2r$}; 
\draw (1.5,5.5) node {$\cdots$};
\draw (2.5,5.5) node {$6$};
\draw (3.5,5.5) node {$4$};
\draw (4.5,5.5) node {$2$};
\end{tikzpicture}
\end{center}
\caption{Linear ordering of the rows and columns in the staircase shape $\sigma_r$.}\label{fig:ferrers labels}
\end{figure}
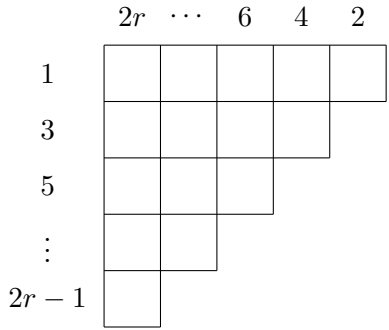

\begin{figure}[htbp]
\begin{center}
\begin{tikzpicture}
\foreach \x in {0,1,2,3}{\fill[black] (0,\x) circle (2pt); \fill[black] (2,\x) circle (2pt);}
\draw (0,0) node[left] {$7$}; \draw (0,1) node[left] {$5$}; \draw (0,2) node[left] {$3$}; \draw (0,3) node[left] {$1$};
\draw (2,0) node[right] {$8$}; \draw (2,1) node[right] {$6$}; \draw (2,2) node[right] {$4$}; \draw (2,3) node[right] {$2$};
\foreach \x in {0,1,2,3}{\draw (0,3) -- (2,\x);}
\foreach \x in {0,1,2}{\draw (0,2) -- (2,\x);}
\foreach \x in {0,1}{\draw (0,1) -- (2,\x);}
\foreach \x in {0}{\draw (0,0) -- (2,\x);}
\end{tikzpicture}
\end{center}
\caption{Labeling of the Ferrers graph $F_4$.}\label{fig:ferrers graph labels}
\end{figure}
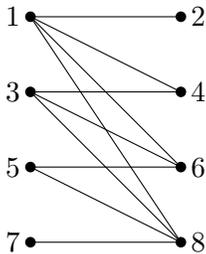

\begin{cor}
For all $r \ge 1$, $\D(F_r) = AE_{2r}$.
\end{cor}

\begin{proof}
As noted earlier, elements of $AE_{2r}$ are derangements whose cycle notations are such that every odd number is followed by something larger and every even number is followed by something smaller.  It is this description which motivates the linear ordering of the vertices of $F_r$, since, as can be seen in the example of Figure~\ref{fig:ferrers graph labels}, each odd vertex is adjacent to all larger even vertices, and each even vertex is adjacent to all smaller odd vertices.

Suppose that $w \in AE_{2r}$.  For any letter $t$ which is not minimal in its cycle in $w$, there are two key facts to note, both of which follow from the above characterization of $AE_{2r}$:
\begin{enumerate}
\item $\lambda_w(t)$ is odd and necessarily smaller than $t$, and
\item $\rho_w(t)$ contains an even number, which is necessarily at least as large as $t$.
\end{enumerate}
In $F_r$, every odd vertex is adjacent to all larger even vertices.  Therefore $\lambda_w(t)$ is adjacent to at least one element of $\rho_w(t)$ in the graph $F_r$.  Hence, Theorem~\ref{thm:criterion} implies that such a $w$ is an element of $\D(F_r)$, and thus $AE_{2r} \subseteq \D(F_r)$.

Now consider $w \in \D(F_r) \setminus AE_{2r}$.  There are two possibilities for the standard cycle form of $w$: either an odd number is followed by something smaller, or an even number is followed by something larger.  Suppose that the first case occurs, and that the odd number is $t$.  By definition, we have $\lambda_w(t) \le t$.  Moreover, since $t$ is followed by a smaller number, we know that $t$ is not the minimal letter in its cycle, so in fact $\lambda_w(t) < t$.  Theorem~\ref{thm:criterion} says that $\lambda_w(t)$ must be adjacent to an element of $\rho_w(t) = \{t\}$ in $F_r$.  However, odd vertices are not adjacent to smaller vertices in $F_r$, so this is a contradiction.  Now suppose that there is an even number $t$ followed by a larger number $s$ in the cycle notation of $w$.  Every element of $\rho_w(s)$ is greater than or equal to $s$, and so strictly greater than $t$.   Theorem~\ref{thm:criterion} indicates that $\lambda_w(s) = t$ must be adjacent to some element of $\rho_w(s)$, but even vertices are not adjacent to larger vertices in $F_r$, so this is a contradiction as well.  Therefore $w \not\in \D(F_r)$, and so $\D(F_r) \subseteq AE_{2r}$.
\end{proof}

\subsection{Coxeter graphs}

In \cite{ragnarsson-tenner}, the boolean numbers were given for the unlabeled Coxeter graphs of the classical finite and affine Coxeter groups.  We now describe linear orderings of the vertices in each graph and the corresponding derangements arising from these orderings.  The most interesting cases are the non-exceptional ones, since there we see patterns in the derangements that explain the previous calculations of $\bn$.

\begin{defn}
Given a string $S = s_1 s_2 \cdots s_n$ of distinct letters, a \emph{valid parsing} of $S$ is a way to partition the letters of $S$ into parts of length at least 2.  We will view the parts as cycles and the product of these cycles as a derangement.
\end{defn}

\begin{example}
There is one valid parsing of $s_1s_2$: $(s_1s_2)$.  Similarly, there is only one valid parsing of $s_1s_2s_3$: $(s_1s_2s_3)$.  There are two valid parsings of $s_1s_2s_3s_4$: $(s_1s_2s_3s_4)$ and $(s_1s_2)(s_3s_4)$.
\end{example}

If no other conditions are imposed, then it is straightforward to show that the number of valid parsings of $s_1s_2 \cdots s_n$ is $f_n$, the $(n-1)$st Fibonacci number.

\begin{defn}
Given a string $S$, let $VP(S)$ denote the set of all valid parsings of $S$.
\end{defn}

Many of the derangements described in this section will be described in terms of valid parsings.  The proofs are fairly straightforward, and details are left to the reader.  We will let the standard Coxeter group names indicate the corresponding unlabeled Coxeter graphs.  For example, the graph $A_n$ consists of an $n$-vertex path.

\begin{defn}
For a path of $n$ vertices, the \emph{path ordering} is the labeling in which the vertex labels increase from one end to the other.
\end{defn}

\begin{cor}\label{cor:coxeter}
\begin{enumerate}
\item Given the path ordering, $\D(A_n) = VP(12\cdots n).$
\item For the graph $D_n$, give the path ordering to one path of length $n-1$, so that the degree-three vertex is labeled $2$; label the remaining leaf $n$.  Then
$$\D(D_n) = \{p \in VP(12n3\cdots(n-2)(n-1)) : 2 \text{ and } n \text{ are in the same cycle of } p\}.$$
\item For the graph $E_n$, give the path ordering to the path of length $n-1$ so that the degree-three vertex is labeled $3$; label the remaining leaf $n$.  Then
$$\D(E_n) = \{p \in VP(123n45\cdots(n-1)) : 3 \text{ and } n \text{ are in the same cycle of } p\}.$$
\end{enumerate}
\end{cor}

The remaining finite Coxeter groups' graphs are all special cases of $A_n$, and their corresponding derangements can be analogously defined using the path ordering.

Results for the affine Coxeter groups are similar, with the only substantially different case being the group $\widetilde{A}_n$, which is a cycle on $n+1$ vertices.

\begin{cor}\label{cor:affine}
Consider the ordering of $V(\widetilde{A}_n)$ where the label $i$ is between the labels $i\pm1$, modulo $n$.  Then
\begin{multline*}
\D(\widetilde{A}_n) = VP(12\cdots n) \ \cup \\
\bigcup_{k=3}^n \big\{p \in VP\big(1k(k+1)\cdots n2 \cdots (k-1)\big) : \{1,k,\ldots,n\} \text{ is in the same cycle of } p\big\}.
\end{multline*}
\end{cor}

Note that the statements in Corollaries~\ref{cor:coxeter} and~\ref{cor:affine} are such that the valid parsings satisfying the given conditions are always written in standard cycle form.

\begin{example}
With the ordering given above,
\begin{align*}
\D(\widetilde{A}_5) = &\{(12345), (123)(45), (12)(345), (13452),\\
&\phantom{\{}(145)(23), (14523), (15)(234), (152)(34), (15234)\}.
\end{align*}
\end{example}

\section*{Acknowledgements} We would like to thank Freyja and Patrek, our other recent collaborations, for allowing us to take the time to finish this paper.

\end{document}